\documentclass[preprint,12pt]{elsarticle}
\usepackage{}
\usepackage{graphics}
\usepackage{graphicx}
\usepackage{subfigure}
\usepackage{latexsym}
\usepackage{amsmath}
\usepackage{graphicx}
\usepackage{times}
\usepackage{graphicx,multicol,enumerate,subfigure,amsmath}

\usepackage{fullpage}
\usepackage{setspace}
\usepackage[english]{babel}
\usepackage[version=3]{mhchem}
\usepackage{amsmath, amssymb, amsthm}
\usepackage{verbatim, latexsym}
\usepackage{colortbl}
\usepackage{multirow}
\usepackage{epsfig,epstopdf}

\newcommand{\norm}[1]{\left\|#1\right\|}

\newtheorem{theorem}{Theorem}
\newtheorem{remark}[theorem]{Remark}

\newtheorem{lemma}[theorem]{Lemma}

\usepackage{soul}
\setstcolor{red}

\biboptions{sort&compress}
\journal{Journal of Our Choice}

\onehalfspace

\begin{document}

\begin{frontmatter}

\title{Randomized Oversampling for Generalized Multiscale Finite Element Methods}
\author{\textbf{Victor M. Calo$^{1,2}$, Yalchin Efendiev}$^{1,3*}$}
\author{\textbf{Juan Galvis$^{4}$, Guanglian Li$^{3}$}}

\address{$^{1}$ Center for Numerical Porous Media (NumPor) \\
  King Abdullah University of Science and Technology (KAUST) \\
  Thuwal 23955-6900, Kingdom of Saudi Arabia.}

\address{$^{2}$ Applied Mathematics \& Computational Science
  and Earth Science \& Engineering\\
  King Abdullah University of Science and Technology (KAUST) \\
  Thuwal 23955-6900, Kingdom of Saudi Arabia.}

\address{$^{3}$ Department of Mathematics \& Institute for Scientific Computation (ISC) \\
  Texas A\&M University \\
  College Station, Texas, USA}

\address{$^{3}$ Departamento de Matem\'{a}ticas\\ Universidad Nacional de Colombia\\
Bogot\'a D.C., Colombia}

\cortext[cor1]{Email address: efendiev@math.tamu.edu}

\begin{abstract}
  In this paper, we study the development of efficient multiscale methods for flows in heterogeneous media.  Our approach uses the Generalized Multiscale Finite Element (GMsFEM) framework. The main idea of GMsFEM is to approximate the solution space locally using a few multiscale basis functions. This is typically achieved by selecting an appropriate  snapshot space and a local spectral decomposition, e.g., the use of oversampled regions in order to achieve  an efficient model reduction.  However, the successful construction of snapshot spaces may be costly if too many local problems need to be solved in order to obtain these spaces. In this paper, we show that  this efficiency can be achieved using a moderate quantity of local solutions (or snapshot vectors) with random boundary conditions on oversampled regions with zero forcing.  Motivated by the randomized algorithm presented in~\cite{ Martinsson06}, we consider a snapshot space which consists of harmonic extensions of random boundary conditions defined in a domain larger than the target region.  Furthermore, we perform an eigenvalue decomposition in this small space. We study the application of  randomized sampling for GMsFEM in conjunction with adaptivity, where local multiscale spaces are adaptively enriched. Convergence analysis is provided.  We present representative numerical results to validate the method proposed.  \end{abstract}


\end{frontmatter}

\section{Introduction}\label{sec:intro}

Model reduction is becoming increasingly important when dealing efficiently with problems characterized by multiple scales. Due to scale disparity, single-scale discretization techniques cannot provide useful results with acceptable computational cost in practice. In order to efficiently handle these multiscale problems, many model reduction techniques have been developed in the literature. These include approaches that are based on homogenization and numerical homogenization~\cite{dur91, weh02,Iliev_LW_10, SISC_2009}, the approaches that employ finite element basis functions to approximate the fine-scale features of the solution space~\cite{Babuska, ArPeWY07, bl11, hw97, ehw99}, and the approaches that employ global model reduction techniques~\cite{GhommemJCP2013, EGG_MultiscaleMOR, ohl12}.  In this paper, our focus is on approaches that are based on multiscale finite element methods which fall in the second category just mentioned.  We use a recently introduced framework known as the Generalized Multiscale Finite Element Method (GMsFEM) and discuss how one can reduce  the setup cost employing randomized Singular Value Decomposition (SVD) concepts~\cite{Martinsson06, hmt11}.

To construct multiscale basis functions, we employ the GMsFEM framework where the multiscale basis functions are constructed via a local spectral decomposition of a snapshot space. This snapshot space typically consists of spatial fields that represent the solution space up to some desired accuracy. For example, one choice for the snapshot space is to use harmonic functions that can represent any boundary value in each coarse region.  These snapshots are constructed by solving local problems for all possible boundary conditions. The latter allows us to incorporate the effects of many small-scale features into these snapshots and thus achieve low dimensional coarse models. However, the computation of these snapshots is expensive.  In this paper, we propose the use of random boundary conditions in constructing snapshot vectors. We show that by using only a few of these randomly generated snapshots, we can adequately approximate dominant modes of the solution space.  To avoid oscillations near the boundary, the oversampling technique is used. More precisely, we solve local problems in domains that are larger than the target coarse blocks. Typically, they are larger by several layers of fine-grid blocks around the target coarse block. Furthermore, we perform a local spectral decomposition using the restriction of the randomly generated snapshots to the target coarse-grid domain.

The use of random boundary conditions (to generate the snapshot spaces) is motivated by the randomized SVD methodology~\cite{Martinsson06, hmt11}. In general, randomized SVD algorithms allow computing dominant eigenvectors by considering a random linear combination of the columns (or rows) of a given matrix.  The random linear combinations typically have a component in the dominant modes and thus, by performing a spectral decomposition in the span of these random combinations, we can achieve an accurate approximation of dominant eigenvectors.

We take advantage of the idea of randomized linear combinations to considerably reduce the computational cost associated with the computation of snapshot vectors.  In particular, we propose solving local problems with random boundary conditions and perform the local spectral decomposition in the space of these snapshots. The cost reduction is due to the fact that, in previous approaches, the snapshot spaces were constructed by solving local problems for every possible boundary condition in each coarse region.  Using our new methodology, the number of snapshots to be generated is  only slightly larger than the number of desired eigenvectors. Our experience suggests that for GMsFEM modeling, in general it suffices to include four additional random boundary conditions to the number of eigenvectors sought.  For instance, in our numerical experiments, when three basis functions per coarse grid are needed, we compute only seven snapshot vectors (i.e., only seven random boundary conditions are generated). We discuss how the number of additional snapshots can depend on the eigenvalue structure for some special cases. This new methodology can provide substantial computational savings in the offline stage as we compute much fewer snapshots.  We show that one needs to use randomized boundary conditions on the oversampled region to avoid oscillations near the boundaries. Indeed, if random boundary conditions are imposed on the target coarse grid (and no oversampling is used), the computed solution has oscillations near the boundaries which can cause large errors. Moreover, oversampling snapshots have several additional advantages~\cite{eglp13oversampling} as they allow faster convergence for GMsFEM discretizations.

We compare the results obtained by using randomized snapshots to  these obtained when all snapshot vectors are used. In the latter, we employ all possible boundary conditions on the oversampled region to construct the snapshot vectors.  The local spectral decomposition is based on local eigenvalue problems, following previous studies~\cite{eglp13oversampling}.  Our numerical results show that one can achieve similar accuracy when using fewer random snapshots instead of using all possible snapshot vectors. Furthermore, we discuss approaches that can improve the results obtained by using randomized snapshots; however, at an additional computational cost.

We analyze the proposed method using~\cite[Lemma 18]{ Martinsson06} and the convergence of oversampling GMsFEM~\cite{ eglp13oversampling}. In a first step, we estimate the approximation error between the full snapshots and randomized snapshots in each coarse neighborhood in a certain norm.  This approximation error is used within GMsFEM analysis to show the convergence of the solution solved in the randomized snapshot space.  We also discuss adaptive strategies for randomized snapshots. In adaptive methods, additional multiscale basis functions are added based on error estimators. These estimators are proposed and investigated in~\cite{ Chung_adaptive14}. Later in the paper, we discuss how additional multiscale basis functions can be computed by considering only a few extra random snapshots. In particular, in simulations we only compute four additional snapshot vectors in order to compute each additional multiscale basis function to be added as a refinement in the coarse domains that contain most error.  The main objective of this paper is to show that the local snapshot spaces can be constructed inexpensively with an accuracy comparable to the state-of-the-art alternatives.

The paper is organized as follows. In Section~\ref{sec:prelim}, we give an introductory description of GMsFEM. In Section~\ref{sec:algorithm}, we present the randomized snapshot algorithm. Section~\ref{sec:numresults} is devoted to numerical results. In this section, we also discuss the use of adaptive strategies and how to compute additional multiscale basis functions. In Section~\ref{sec:analysis}, we present the mathematical analysis of the method and in Section~\ref{sec:conclusions} we draw conclusions.

\section{Preliminaries}
\label{sec:prelim}

We consider linear elliptic equations of the form
\begin{equation} \label{eq:original}
-\mbox{div} \big( \kappa(x) \, \nabla u \big)=f \, \, \text{in} \, D,
\end{equation}
where $u$ is prescribed on $\partial D$. We assume that the coefficient $\kappa(x)$ has multiple scales and high variations (e.g., see Fig.~\ref{fig:perms}). In this paper we focus on the two dimensional case but our methodology can be easily extended to problems in three dimensions, where the implied savings could be larger.

\begin{figure}[htb!]
  \centering
  \subfigure[$\kappa_1(x)$]{\label{fig:perm_cross}
    \includegraphics[width = 0.48\textwidth]{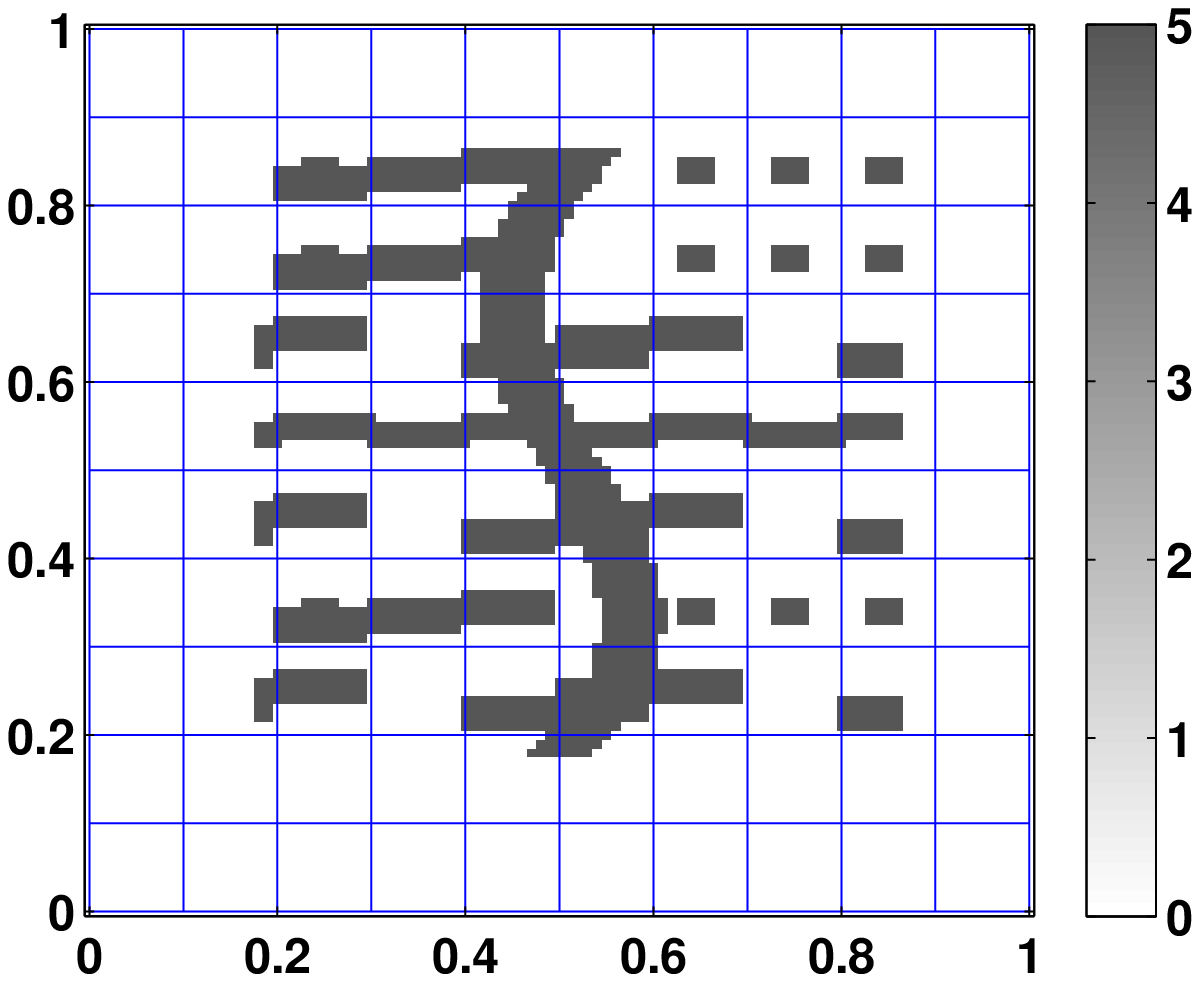}
   }
  \subfigure[$\kappa_2(x)$]{\label{fig:perm_hcc}
     \includegraphics[width = 0.48\textwidth]{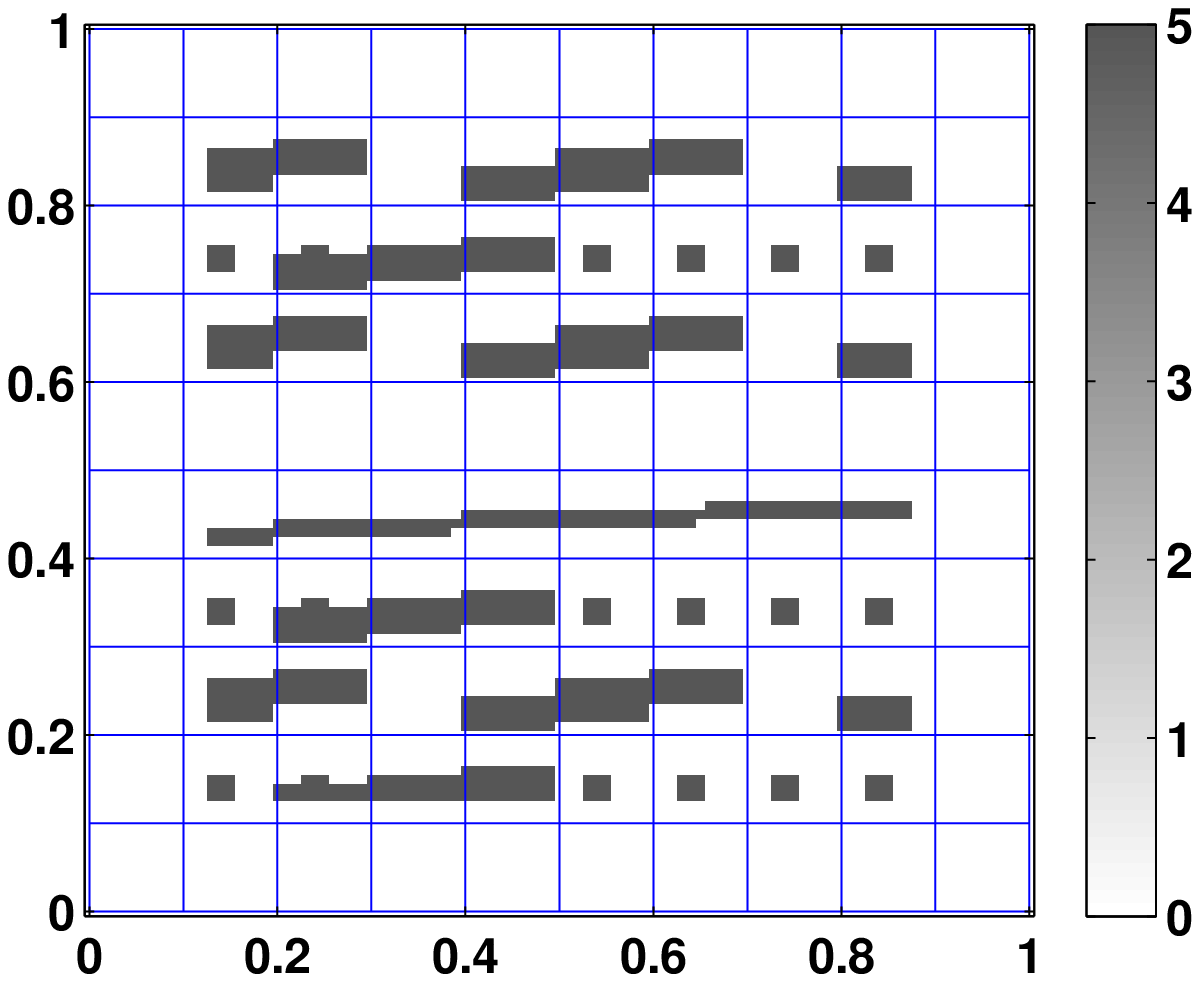}
  }
 \caption{Permeability fields in $\log_{10}$-scale.} 
\label{fig:perms}
\end{figure}

\subsection{Fine and coarse grids}

Let $\mathcal{T}^H$ be a conforming partition of the computational domain $D$ into finite elements denoted by $\{K_j\}$ (triangles, quadrilaterals, tetrahedrals, etc.), called coarse grid. Assume that each coarse subregion is partitioned into a connected union of fine-grid blocks. Assume the fine grids match across coarse elements boundaries and denote by $\mathcal{T}^h$ the obtained (fine-grid) triangulation of $D$.  We use $\{x_i\}_{i=1}^{N_c}$ (where $N_c$ the number of coarse nodes) to denote the vertices of the coarse mesh $\mathcal{T}^H,$ and define the neighborhood of the node $x_i$ by
\begin{equation} \label{neighborhood}
\omega_i=\bigcup\{ K_j\in\mathcal{T}^H; ~~~ x_i\in \overline{K}_j\}.
\end{equation}
See Fig.~\ref{schematic} for an illustration of neighborhoods and elements subordinated to the coarse discretization.  We introduce notation for oversampled regions. We denote by $\omega_i^{+}$ the oversampled region of $\omega_i\subset \omega_i^{+}$, defined by adding several fine- or coarse-grid layers around $\omega_i$.  We emphasize that the coarse-grid is too coarse to effectively resolve all heterogeneities and scales present in the coefficient $\kappa$, while the fine grid resolves all variations of $\kappa$ but it leads to a huge linear system that is not practical to solve.

\begin{figure}[tb]
  \centering
  \includegraphics[width=1.0 \textwidth]{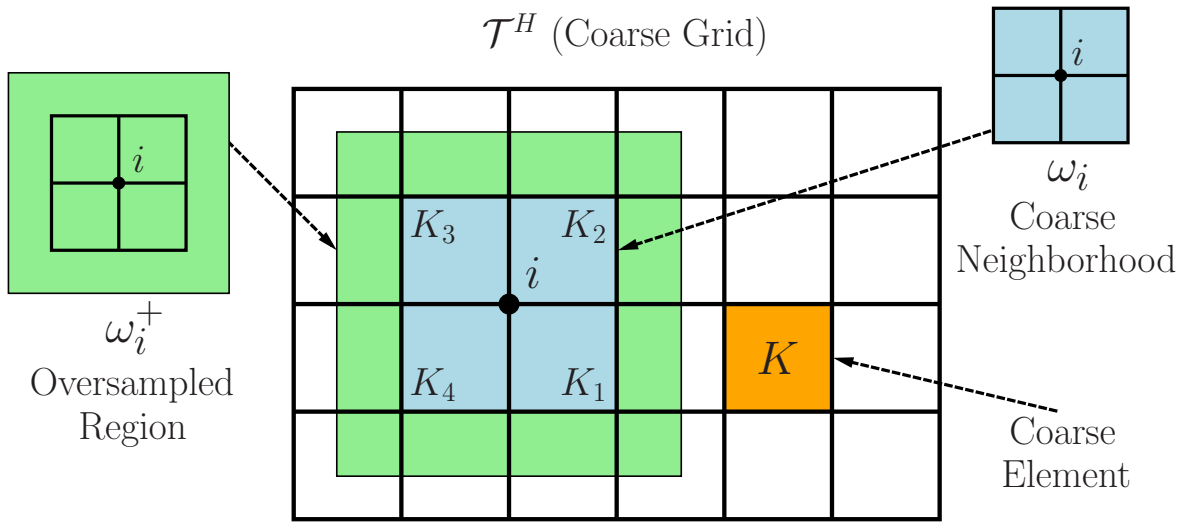}
  \caption{Illustration of a coarse neighborhood and oversampled domain. Here, $K$ is a coarse-grid block, $\omega_i$ is a coarse neighborhood of $x_i$, and
$\omega_i^+$ is an oversampled region}
  \label{schematic}
\end{figure}

\subsection{Generalized Multiscale Finite Element Method (GMsFEM)}

Throughout this paper, we use the continuous Galerkin formulation, and use $\omega_i$ as the support of basis functions. The regions $\omega_i^{+}$ are used to construct the multiscale basis functions.  For the purpose of this description, we formally denote the basis functions of the offline space $V_{\text{off}}$ by $\phi_k^{\omega_i}$.  The solution is sought as $u_H(x)=\sum_{i,k} c_{k}^i \phi_{k}^{\omega_i}(x)$, where $k$ denotes the basis function index in the domain $\omega_i.$  Once the basis functions are identified, we solve
\begin{equation}
\label{eq:globalG} a(u_H,v)=(f,v), \quad \text{for all} \, \, v\in
V_{\text{off}},
\end{equation}
and
\[
a(u,v)=\int_D \kappa(x)\nabla u \cdot \nabla v.
\]

Now, we briefly describe GMsFEM. We consider oversampling for GMsFEM (see~\cite{eglp13oversampling, egh12}) that uses harmonic snapshots. That is, snapshots vector are obtained as harmonic extensions of some subset of all possible boundary conditions on the oversampled domain.  We construct a snapshot space $V_{\text{snap}}^{\omega_i^+}$.  Construction of the snapshot space involves solving local problems and we detail the standard process below~\cite{eglp13oversampling, egh12}.

The snapshot space consists of harmonic extensions of fine-grid functions defined on the boundary of $\omega_i^{+}$. More precisely, for each fine-scale function with support on the boundary of the oversampled coarse domain, $\delta_l^h(x)$, we solve a local problem. Let $\delta_l^h(x_k)=\delta_{lk}$ be one of these functions where for all $l,k\in \textsl{J}_{h}(\omega_i^{+})$, where $\textsl{J}_{h}(\omega_i^{+})$ is the fine-grid boundary nodes on $\partial\omega_i^{+}$ and $\delta_{lk}$ is Kronecker's delta with value 1 for $k=l$ and value 0 otherwise. Thus, the local problem to solve is
\begin{equation}
\label{eqn:snapshot}
-\mbox{div}(\kappa (x)\nabla  \psi_{l,\omega_i}^{+, \text{snap}})=0\ \ \text{in} \ \omega_i^{+}
\end{equation}
subject to boundary condition, $ \psi_{l,\omega_i}^{+, \text{snap}}=\delta_l^h(x)$ on $\partial\omega_i^{+}$.  We form  the snapshot matrices by placing the solutions of these local problems as the rows of this matrix (throughout, for notational convenience, we do not distinguish between the fine-grid vectors and their continuous representations)
\[
\Psi_{\omega_i}^{+,\text{snap}}=[\psi_{1,\omega_i}^{+, \text{snap}};...;\psi_{l,\omega_i}^{+, \text{snap}};....].
\]
We define the vectors $\psi_{l,\omega_i}^{\text{snap}}$ as the restrictions of the snapshot vectors $\psi_{1,\omega_i}^{+, \text{snap}}$ to degrees of freedom in $\omega_i$ by taking their values at the fine-grid nodes of $\omega_i$. Considering these vectors, we form the snapshot matrix in $\omega_i$
\begin{equation} \label{eq:snapdef} \Psi_{\omega_i}^{\text{snap}}=[\psi_{1,\omega_i}^{\text{snap}};...;\psi_{l,\omega_i}^{\text{snap}};....].  \end{equation}

Next, we discuss the construction of a smaller offline space using an eigenvalue problem~\cite{egh12}.  In order to construct an offline space $V_{\text{off}}$, we  reduce the dimension of the snapshot space using an auxiliary spectral decomposition.  We seek a subspace of the snapshot space  where to approximate any element of the snapshot space in the appropriate norm defined via the following auxiliary bilinear forms.  For each $\omega_i$, we define
\begin{eqnarray}
A^{\text{off}} \Theta_k^{\text{off}} &=& \lambda_k^{\text{off}} S^{\text{off}}\Theta_k^{\text{off}},  \label{offeig1}
\end{eqnarray}
where
\begin{equation*}
 \displaystyle A^{\text{off}}= [a^{\text{off}}_{mn}] = \int_{\omega_i} {\kappa}(x) \nabla \psi_{m,\omega_i}^{+,\text{snap}} \cdot \nabla \psi_{n,\omega_i}^{+,\text{snap}} = \Psi_{\omega_i}^{\text{snap}} {A}  (\Psi_{\omega_i}^{\text{snap}})^T
 \end{equation*}
and
\begin{equation*}
 \displaystyle S^{\text{off}} = [s^{\text{off}}_{mn}] = \int_{\omega_i} {\widetilde{\kappa}}(x)  \psi_{m,\omega_i}^{+,\text{snap}}  \psi_{n,\omega_i}^{+,\text{snap}} = \Psi_{\omega_i}^{\text{snap}} {S}  (\Psi_{\omega_i}^{\text{snap}})^T.
 \end{equation*}
 The coefficient $\widetilde{{\kappa}}(x)$ uses multiscale partition of unity functions (cf.,~\cite{egh12}) which is described in~\eqref{def:tildekappa}. Here, $A$ and $S$ are fine-grid stiffness and mass matrices in the coarse region.  To generate the offline space, we then choose the smallest $M_{\text{off}}$ eigenvalues of Eqn.~\eqref{offeig1} for each $\omega_i^+$ and form the corresponding eigenvectors in the respective space of snapshots by setting $\psi_{k,\omega_i}^{+,\text{off}} = \sum_j \Theta_{kj}^{\text{off}} \psi_{j,\omega_i}^{+,\text{snap}}$ (for $k=1,\ldots, M_{\text{off}}$), where $\Theta_{kj}^{\text{off}}$ are the components of the vector $\Theta_{k}^{\text{off}}$. We then create the offline matrices
 $$
\Psi_{\omega_i}^{+,\text{off}} = \left[ \psi_{1,\omega_i}^{+,\text{off}}, \ldots, \psi_{\omega_i,M_{\text{off}}}^{+,\text{off}} \right]
\quad \text{and} \quad \Psi_{\omega_i}^{\text{off}}= \left[ \psi_{1,\omega_i}^{\text{off}}, \ldots, \psi_{M_{\text{off}},\omega_i}^{\text{off}} \right],
$$
where $\psi_{k,\omega_i}^{\text{off}}$ is the restriction of $\psi_{k,\omega_i}^{+,\text{off}}$ to $\omega_i$.  To construct multiscale basis functions, we multiply the dominant eigenvectors by a partition of unity functions $\chi_i$ that are supported in $\omega_i$, such that $\sum_i \chi_i=1$. More precisely, the offline space is composed of the following basis functions,
\begin{align}\label{eqn:global_offline}
\phi_k^{\omega_i}=\chi_i \psi_k^{\omega_i}.
\end{align}
We can choose the partition of unity functions to be multiscale finite element basis functions; see~\cite{eh09}. Let $\chi_i^0$ be the nodal basis of the standard finite element space $W_H$.  For example, $W_H$ consists of piecewise linear functions if ${\cal T}_H$ is a triangular partition or $W_H$ consists of piecewise bi-linear functions if ${\cal T}_H$ is a rectangular partition.``Standard'' multiscale finite element basis functions coincide with $\chi_i^0$ on the boundaries of the coarse partition and satisfy:
\begin{eqnarray}
\mbox{div}(\kappa\nabla\chi_i^{ms})=0\ \ \mbox{in }K\in \omega_i,\quad
\chi_i^{ms}=\chi_i^0\ \ \mbox{in }\partial K,\ \ \mbox{ for all }\  K\in \omega_i,
\label{e}
\end{eqnarray}
where $K$ is a coarse grid block within $\omega_i$. In our numerical implementations, we take $\widetilde{\kappa}=\kappa$ for the computation of mass matrix. However, one can take a weighted permeability field (see detailed discussion in~\cite{egh12}) such as
\begin{equation}
\label{def:tildekappa}
\widetilde{\kappa} =\sum_i \kappa |\nabla \chi_i^{+}|^2.
\end{equation}

\section{Randomized Oversampling}
\label{sec:algorithm}

As described above, a usual choice for the snapshot space consists of the harmonic extension of fine-grid functions defined on the boundary of $\omega_i^{+}$.  This type of snapshot is complete in the sense that it captures all the boundary information of the solution. However, the computational cost is expensive since, in each local coarse neighborhood, $O(n^{\omega_i^+})$ number of local problems is required to solve. Here, $n^{\omega_i^+}$ denotes the number of fine grids on the boundary of $\omega_i^+$.  A smaller yet accurate snapshot space is needed to build a more efficient multiscale method.

In the following, we generate inexpensive snapshots using random boundary conditions. That is, instead of solving Eqn.~\eqref{eqn:snapshot} for each fine boundary node, we solve a small number of local problems imposed with random boundary conditions:
 \begin{align}
 \label{eq:random bc}
  \psi_{l,\omega_i}^{+, \text{rsnap}}=r_{l} \text{ on } \partial\omega_i^{+},
 \end{align}
 where $r_{l}$ are independent identically distributed (i.i.d.) standard Gaussian random vectors on the fine-grid nodes of the boundary. Then, we can obtain the local random snapshot on the target domain $\omega_i$ by restricting the solution of this local problem, $\psi_{l,\omega_i}^{+, \text{rsnap}}$ to $\omega_i$ (which is denoted by $\psi_{l,\omega_i}^{ \text{rsnap}}$). The space generated by $\psi_{l,\omega_i}^{ \text{rsnap}}$ is a subspace of the space generated by all local snapshots $\Psi_{\omega_i}^{\text{snap}}$. Therefore, there exists a randomized matrix $\mathcal{R}$ with rows composed by the random boundary vectors $r_{l}$, such that,
\begin{align}\label{eqn:random_snapshots}
 \Psi_{\omega_i}^{\text{rsnap}}=\mathcal{R}\Psi_{\omega_i}^{\text{snap}}.
 \end{align}
 Using these snapshots, we follow the procedure in the previous section to generate multiscale basis functions. Below, we summarize the algorithm. We denote the buffer number $p_{\text{bf}}^{\omega_i}$ for each $\omega_i$ and the number of local basis functions by $k_{\text{nb}}^{\omega_i}$ for each $\omega_i$. Later on, we use the same buffer number for all $\omega_i$ and simply use the notation $p_{\text{bf}}$.
 \begin{table*}[h]\label{algorithm:random_snapshots}
 \caption{Randomized GMsFEM Algorithm}
 \begin{tabular}{r l}
 \hline\hline
 \\
    \textbf{Input}:& Fine grid size $h$, coarse grid size $H$,
oversampling size $t$, buffer number $p_{\text{bf}}^{\omega_i}$ for each $\omega_i$, \\
    & the number of local basis functions $k_{\text{nb}}^{\omega_i}$ for each $\omega_i$;\\
    \textbf{output}: & Coarse-scale solution $u_H$.\\

  1.& Generate oversampling region for each coarse block: $\mathcal{T}^H$, $\mathcal{T}^h$, and $\omega_i^{+}$; \\

  2.& Generate $k_{\text{nb}}^{\omega_i}+p_{\text{bf}}^{\omega_i}$ random vectors $r_l$ and obtain randomized snapshots in $\omega_i^{+}$ (Eqn.~\eqref{eq:random bc});  \\

& Add a snapshot that represents the constant function on $\omega_i^+$;\\
  3. & Obtain $k_{\text{nb}}^{\omega_i}$ offline basis by a spectral decomposition (Eqn.~\eqref{offeig1} restricted to random snapshots); \\

4. & Construct multiscale basis functions (Eqn.~\eqref{eqn:global_offline}) and solve (Eqn.~\eqref{eq:globalG} ).\\

 \hline\hline
 \end{tabular}
 \end{table*}
%


\section{Numerical results}
\label{sec:numresults}

In this section, we present representative numerical experiments that demonstrate the good performance of the randomized snapshots algorithm. We take the domain $D$ as a square, set the forcing term $f=0$ and use a linear boundary condition for the problem~\eqref{eq:original},  that is, $u=x_1+x_2$ on $\partial D$ where $x_i$ are the Cartesian components of each point.  In our numerical simulations, we use a  coarse grid of $10 \times 10$ blocks, and each coarse grid block is divided into $10\times 10$ fine grid blocks. Thus, the whole computational domain is partitioned by a $100 \times 100$ fine grid. We use a few multiscale basis functions per coarse block. These coarse basis set defines the problem size.  We assume that the fine-scale solution is obtained by discretizing problem~\eqref{eq:original} by the classical conforming piecewise bilinear elements on the fine grid.  To test the performance of our algorithm, we consider two permeability fields $\kappa$ as depicted in Figure~\ref{fig:perms}.  The first permeability field (left figure) has more connected regions and they are more irregular compared to the second permeability field (right figure).  We observed similar behavior for these two cases, and therefore we focus on the numerical results for the first permeability field (Figure~\ref{fig:perm_cross}).

In Table~\ref{table:Reverse of HCCResult}, a comparison between using all snapshots and the randomized snapshots is shown.  The first column shows the dimension of the offline space for each test. We choose $5$, $10$, $15$, $20$, and $25$ basis functions per each interior node (in addition to the constant eigenvectors) and use an oversampling layer that consists of   three fine-grid blocks ($t=3$). The offline space $V_{\text{off}}$ is defined via a local spectral decomposition as specified in Section~\ref{sec:algorithm}. The snapshot ratio is calculated as the number of randomized snapshots divided by the number of the full  snapshots. This ratio is displayed in the second column.  Here, the total number of snapshots refers to the number of boundary nodes of the oversampled region. In our numerical results, an oversampled region has $26\times 26$ fine-grid dimension and there are total $104$ snapshots if all boundary nodes are used. For example, when the dimension of the offline space is $931$, we only compute $14$ snapshots instead of $104$. This ratio gives the information on the computational savings of our algorithm compared to the previous algorithm using all snapshots. The next two columns shows the relative weighted $L^2$ error and relative energy error using the full snapshots. The weighted $L^2$ norm and energy norm are defined as
\[
\|u\|_{L^2_\kappa}=\left(\int_D \kappa u^2\right)^{1\over 2} \quad \mbox{ and } \quad
\ \ \|u\|_{H^1_\kappa}=\left(\int_D \kappa |\nabla u|^2\right)^{1\over 2},
\]
respectively. Further, the relative weighted $L^2$ error and relative energy error using the randomized snapshots are shown in the last two columns.  From this table, we observe that the randomized algorithm converges in the sense that the relative error decreases as we increase the dimension of the coarse space. Comparing the fourth column with the last column, we conclude that the accuracy when using the randomized snapshots is similar to using all snapshot vectors. The latter has much larger dimension as shown in the second column  that shows the percentage of the snapshots computed.  Therefore, the proposed method is an order of magnitude faster while having comparable accuracy. For example, when the dimension of the offline space is $931$, the accuracy of the methods is comparable while randomized snapshot approach uses only $13.46$\% of the snapshots.   Similar results are obtained when the fine mesh is refined to $200\times 200$. In particular, with the offline space with the dimension $931$ and the snapshot ratio of $10$\%, we obtain similar $L^{2}_\kappa(D)$ and $H^{1}_\kappa(D)$ errors which are $1.28$\% and $24.02$\%.   The behavior is similar when we use the permeability field in Fig.~\ref{fig:perm_hcc}. The results are displayed in Table~\ref{table:Reverse of HCCResult1}.  Here, $p_{\text{bf}}$ refers to the buffer that is used to compute the eigenvectors. For example, $p_{\text{bf}}=4$ means that we use $n+4$ snapshots to compute $n$ basis functions for each coarse block.

\begin{table}[htb!]
\centering
\caption{Numerical results comparing the results between using all harmonic snapshots and the snapshots generated by random boundary conditions with $p_{\text{bf}}=4$, $\kappa$ as shown in Fig.~\ref{fig:perm_cross}.  In the parenthesis, we show a higher value of the snapshot ratio.}
\label{table:Reverse of HCCResult}
\small
\begin{tabular}{|r|c|c|c|c|c|c|c}
\hline
\multirow{2}{*}{$\text{dim}(V_{\text{off}})$}  &
\multirow{2}{*}{Snapshot ratio (\%)}  &
\multicolumn{2}{c|}{   All snapshots (\%) } &
\multicolumn{2}{c|}{   Few randomized snapshots (\%) }\\
\cline{3-6} {}&{}&
$\hspace*{0.8cm}   L^{2}_\kappa(D)   \hspace*{0.8cm}$ &
$\hspace*{0.8cm}   H^{1}_\kappa(D)  \hspace*{0.8cm}$
&
$\hspace*{0.8cm}   L^{2}_\kappa(D)   \hspace*{0.8cm}$ &
$\hspace*{0.8cm}   H^{1}_\kappa(D)  \hspace*{0.8cm}$
\\
\hline\hline
       $526$ &$8.65 (15.38)$  &  $0.87$    & $18.15$&  $2.81(1.38)$    & $44.95 (26.04)$   \\
      $931$ &$13.46$   &    $0.64$    & $14.85$ &    $1.04$    & $23.61$\\
      $1336$ &$18.27$   &    $0.55$    & $13.59$ &    $0.70$    & $18.08$\\
       $1741$&$23.08$   &   $0.50$    & $12.69$ &    $0.64$    & $15.91$\\
       $2146$&$27.88$   &   --  &--&$0.54$&$14.16$  \\
\hline
\end{tabular}
\end{table}

In Fig.~\ref{fig:solns}, the fine-scale solution, coarse-scale solution using all snapshots and coarse-scale solution using randomized snapshots are  shown.  They are obtained using the second test (when the dimension of the offline space is $931$) in Table~\ref{table:Reverse of HCCResult}.  These two coarse-scale solutions are a good approximation of the fine-scale solution. This is corroborated in Fig.~\ref{fig:errors}, where we plot the absolute error of the two solutions.

\begin{figure}[htb!]
  \centering
  \subfigure[Fine-scale solution.]{\label{fig:perm_cross1}
    \includegraphics[width = 0.315\textwidth]{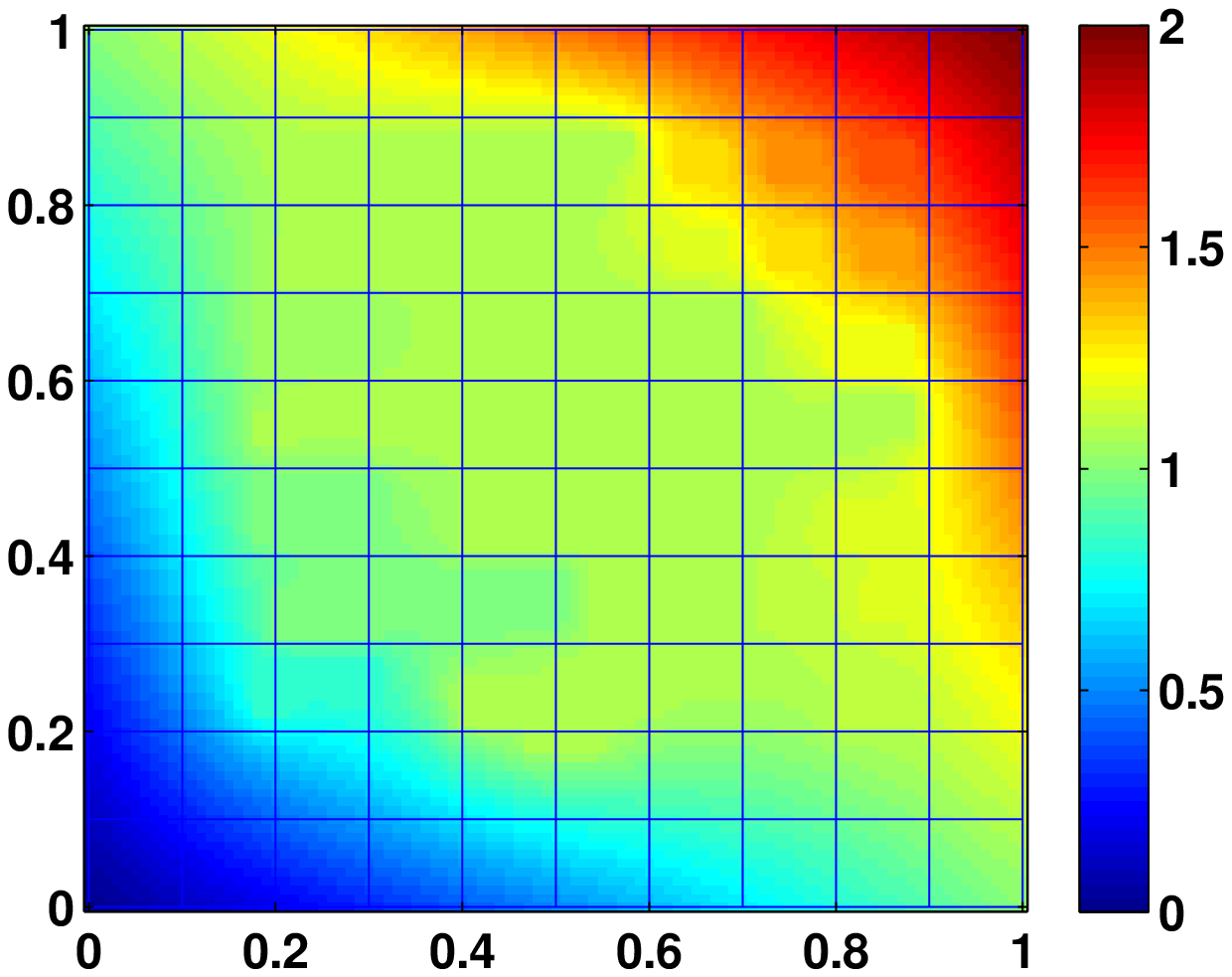}
   }
   \subfigure[coarse-scale solution using the full snapshots]{\label{fig:umsfineSS}
    \includegraphics[width = 0.315\textwidth]{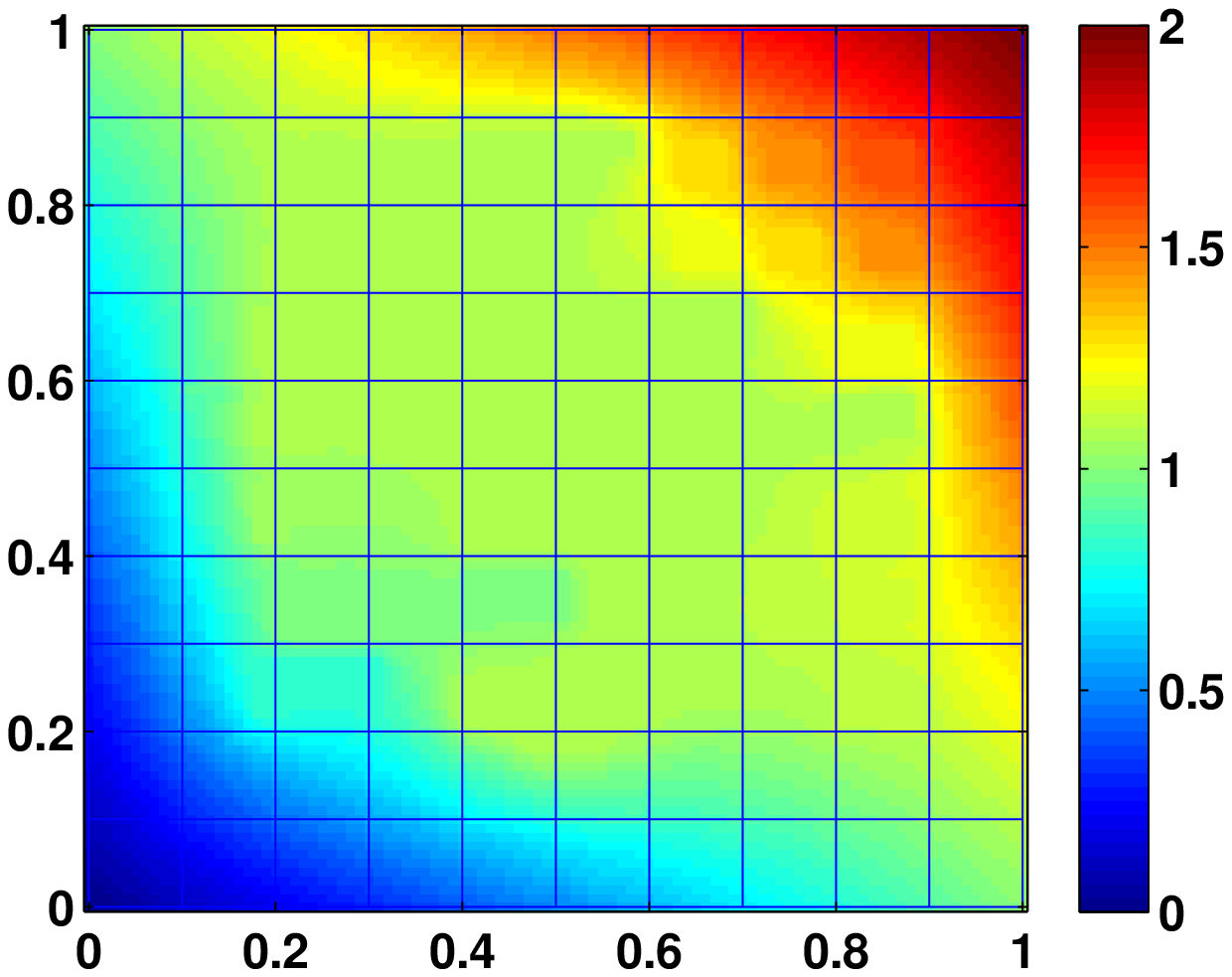}
   }
  \subfigure[coarse-scale solution using the randomized snapshots]{\label{fig:umsfineSS_random}
     \includegraphics[width = 0.315\textwidth]{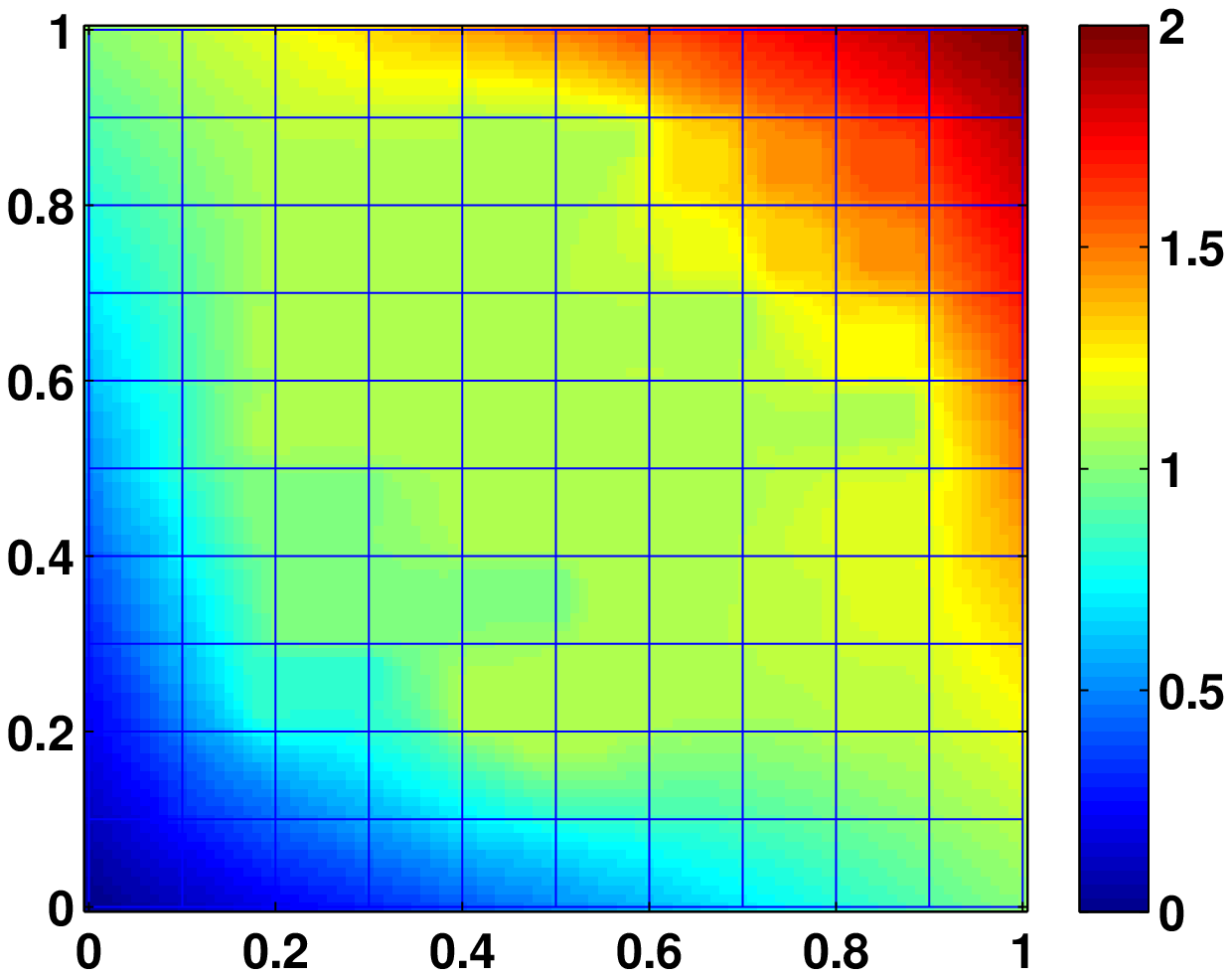}
  }
 \caption{The fine-scale solution and coarse-scale solutions correspond to Fig.~\ref{fig:perm_cross}. }\label{fig:solns}
\end{figure}
\begin{figure}[htb!]
  \centering
   \subfigure[Absolute error using the full snapshots]{\label{fig:error_full}
    \includegraphics[width = 0.45\textwidth]{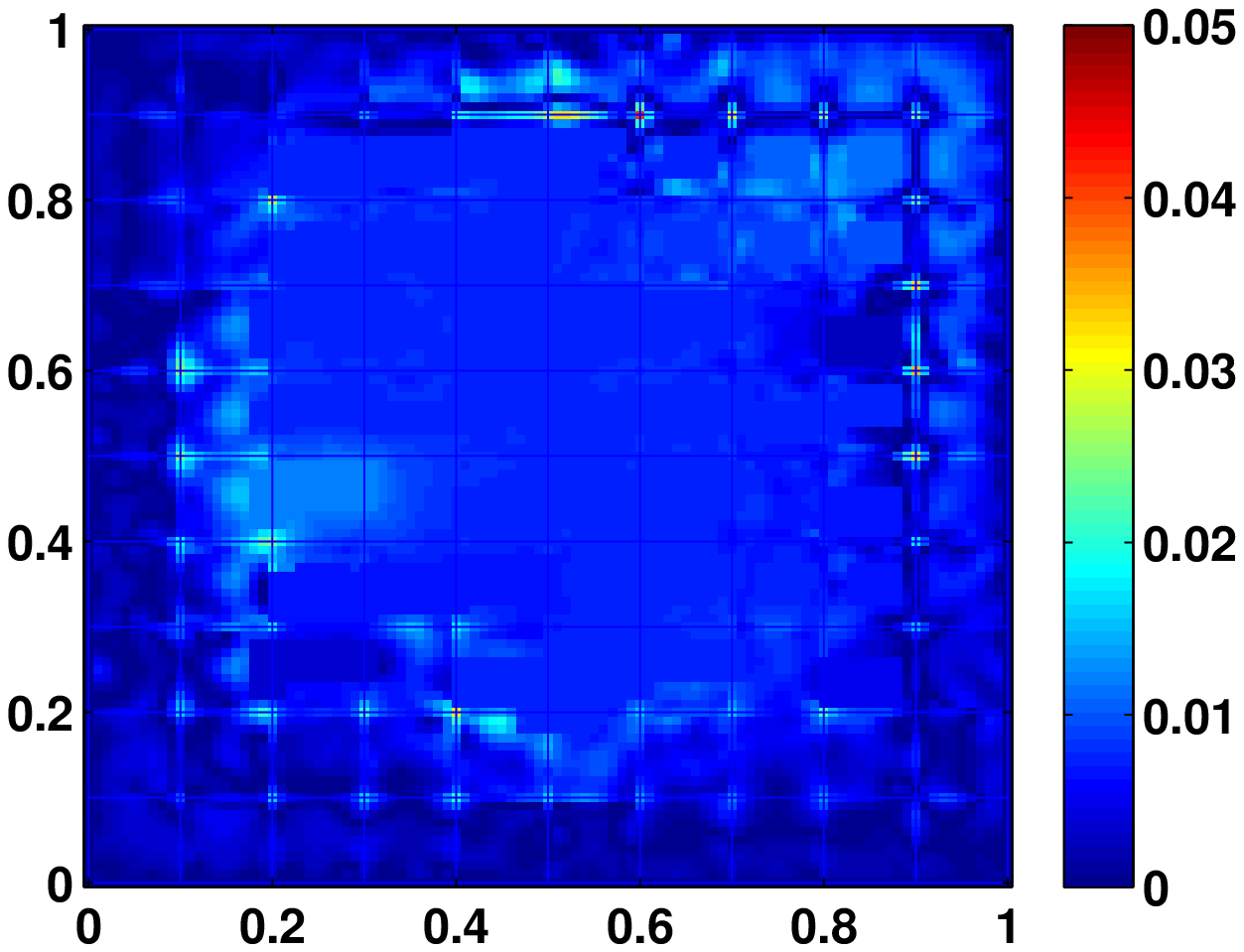}
   }
  \subfigure[Absolute error using the randomized snapshots]{\label{fig:error_random}
     \includegraphics[width = 0.45\textwidth]{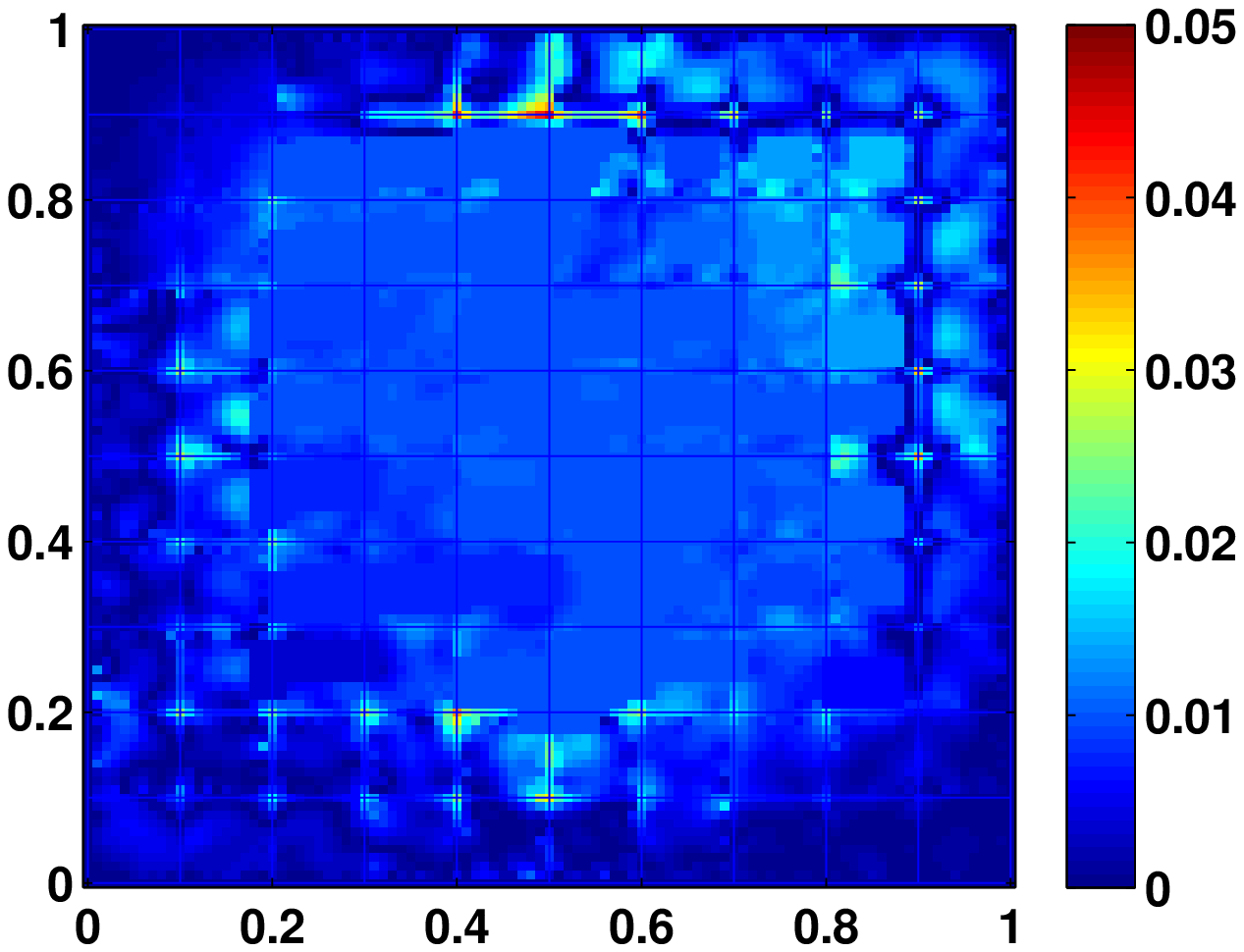}
  }
 \caption{The absolute errors correspond to Fig.~\ref{fig:perm_cross} using full snapshots and random snapshots. }\label{fig:errors}
\end{figure}
\begin{table}[htb!]
\centering
\caption{Numerical results comparing the results between using all harmonic snapshots and the snapshots generated by random boundary conditions with $p_{\text{bf}}=4$, $\kappa$ as shown in Fig.~\ref{fig:perm_hcc}.}
\label{table:Reverse of HCCResult1}
\small
\begin{tabular}{|r|c|c|c|c|c|c|c}
\hline
\multirow{2}{*}{$\text{dim}(V_{\text{off}})$}  &
\multirow{2}{*}{snapshot ratio (\%)}  &
\multicolumn{2}{c|}{ all snapshots (\%) } &
\multicolumn{2}{c|}{  using the randomized snapshots (\%) } \\
\cline{3-6} {}&{}&
$\hspace*{0.8cm}   L^{2}_\kappa(D)   \hspace*{0.8cm}$ &
$\hspace*{0.8cm}   H^{1}_\kappa(D)  \hspace*{0.8cm}$
&
$\hspace*{0.8cm}   L^{2}_\kappa(D)   \hspace*{0.8cm}$ &
$\hspace*{0.8cm}   H^{1}_\kappa(D)  \hspace*{0.8cm}$
\\
\hline\hline
       $526$ &$8.65 (15.38)$    &  $0.71$    & $20.98$&  $1.33(0.80)$    & $33.76(24.14)$  \\
      $931$ &$13.46$   &    $0.51$    & $17.33$ &    $0.66$    & $21.67$\\
      $1336$ &$18.27$   &    $0.45$    & $15.83$ &    $0.53$    & $18.26$\\
       $1741$&$23.08$   &   $0.40$    & $14.66$ &    $0.48$    & $17.13$\\
       $2146$&$23.88$   &   --  &--&$0.43$&$15.39$  \\

\hline
\end{tabular}
\end{table}%
Next, we investigate the effect of the buffer number $p_{\text{bf}}$ on the accuracy of the coarse solution. We test a series of simulations with different $p_{\text{bf}}$ while keeping the   coefficients and meshes fixed. The results are presented in Table~\ref{table:bufferEffect}, which shows that a larger buffer coefficient   decreases the relative energy error. However, there is no need for very large values.  If we take $p_{\text{bf}}=4$, we can get a coarse solution with error of $15.51\%$, while obtaining a $14.49\%$ error if using $p_{\text{bf}}=20$ at the cost of solving 16 extra local problems for each inner coarse node.

\begin{table}[htb!]
\centering
\caption{Numerical results for different $p_{\text{bf}}$ and using 20 local basis in each coarse neighborhood, $\kappa$ as shown in Fig.~\ref{fig:perm_cross}.}
 \label{table:bufferEffect}
\begin{tabular}{|r|c|c|c|c|}
\hline
\multirow{2}{*}{$p_{\text{bf}}$}  &
\multicolumn{2}{c|}{  $\|u-u^{\text{off}} \|$ (\%) }  \\
\cline{2-3} {}&
$\hspace*{0.8cm}   L^{2}_\kappa(D)   \hspace*{0.8cm}$ &
$\hspace*{0.8cm}   H^{1}_\kappa(D)  \hspace*{0.8cm}$
\\
\hline\hline
       $4$     &  $0.62$    & $15.51$  \\
      $10$    &    $0.62$    & $15.08$ \\
      $15$    &     $0.57$  &$14.70$ \\
       $20$   &   $0.57$  &$14.49$  \\

\hline
\end{tabular}
\end{table}
%

\begin{table}[htb!]
\centering
\caption{Numerical results for different oversampling domain $\omega_i^{+}=\omega_i + t$ and using 20 local basis in each coarse neighborhood, $p_{\text{bf}}=4$, $\kappa$ as shown in Fig.~\ref{fig:perm_cross}.}
 \label{table:oversamplingEffect}
\begin{tabular}{|c|c|c|c|c|}
\hline
\multirow{2}{*}{$t$}  &
\multicolumn{2}{c|}{  $\|u-u^{\text{off}} \|$ (\%) }  \\
\cline{2-3} {}&
$\hspace*{0.8cm}   L^{2}_\kappa(D)   \hspace*{0.8cm}$ &
$\hspace*{0.8cm}   H^{1}_\kappa(D)  \hspace*{0.8cm}$
\\
\hline\hline
       $0$     &  $1.52$    & $23.26$  \\
      $2$    &    $0.61$    & $15.63$ \\
      $4$    &    $0.62$    & $15.56$ \\
      $7$    &     $0.59$  &$15.24$ \\
\hline
\end{tabular}
\end{table}
Lastly, numerical tests are conducted to study the influence of oversampling effects on the accuracy of the randomized snapshots. The simulation results are shown in Table~\ref{table:oversamplingEffect}. From this table, we observe that oversampling technique is needed to obtain an accurate solution. However, a larger oversampling domain is not necessary since it increases the computational cost of the solution, while no significant improvement in the solution accuracy is observed.

\subsection{Comparison of results of different spectral problems}

As we mentioned in the introduction, Section~\ref{sec:intro}, one can use solution-based boundary conditions to achieve higher accuracy compared to the random boundary conditions.  In this section, we demonstrate this.  The main idea behind this algorithm is to select boundary modes using a small spectral decomposition over the boundary layer $\mathcal{L}_i$ instead of the oversampling region $\omega_i^{+}$ that surrounds the boundary in the spectral problem Eqn.~\eqref{offeig1}. More precisely, we consider a local spectral problem in the layer of a few fine-grid blocks in the region that contains the boundary of $\omega_i$ (see Fig.~\ref{schematic_BL}). We choose a layer that has a thickness  of five fine-grid elements (two interior to $\omega_i$ and three on the immediate neighborhood of $\omega_i$). Furthermore, we select dominant eigenvectors (corresponding to smallest eigenvalues) by solving local eigenvalue problem in the strip. The local eigenvalue problem uses local stiffness and mass matrices (as in~\cite{ ge09_1,egw10}). This approach provides correct fine-scale features and we expect higher accuracy compared to the randomized snapshots.
\begin{figure}[tb]
  \centering
  \includegraphics[width=1.0 \textwidth]{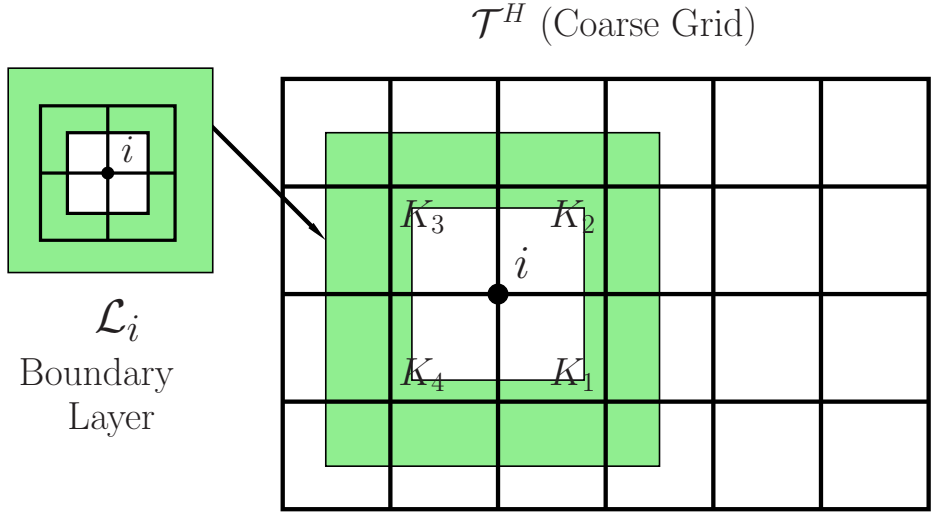}
  \caption{Illustration of a skin layer $\mathcal{L}_i$ that is used for computing boundary conditions for the snapshots in $\omega$.}
  \label{schematic_BL}
\end{figure}
The numerical results are shown in Table~\ref{table:skin_snapshots}.  Comparing the fourth column with the last column of Table~\ref{table:skin_snapshots}, we observe that this new algorithm is more accurate compared to the previous one. Taking the fifth row as an example, for the same dimension of the offline space, the new algorithm gives $14.97\%$ error while the previous algorithm ends with $17.13\%$. In general, one can apply randomized snapshot algorithms to reduce the computational cost associated with our new algorithm. That is, one can use randomized snapshots for the strip $\mathcal{L}_i$ to reduce the computational cost further.

\begin{table}[htb!]
\centering
\caption{Numerical results comparing the results between the snapshots obtained from  skin layer spectral problems and the snapshots generated by random boundaries with $p_{\text{bf}}=4$, $\kappa$ as shown in Fig.~\ref{fig:perm_hcc}. }
\label{table:skin_snapshots}
\small
\begin{tabular}{|r|c|c|c|c|c|c|c}
\hline
\multirow{2}{*}{$\text{dim}(V_{\text{off}})$}  &
\multirow{2}{*}{snapshot ratio (\%)}  &
\multicolumn{2}{c|}{ snapshots from skin layer (\%) } &
\multicolumn{2}{c|}{ randomized snapshots (\%) } \\
\cline{3-6} {}&{}&
$\hspace*{0.8cm}   L^{2}_\kappa(D)   \hspace*{0.8cm}$ &
$\hspace*{0.8cm}   H^{1}_\kappa(D)  \hspace*{0.8cm}$
&
$\hspace*{0.8cm}   L^{2}_\kappa(D)   \hspace*{0.8cm}$ &
$\hspace*{0.8cm}   H^{1}_\kappa(D)  \hspace*{0.8cm}$
\\
\hline\hline
       $526$ &$8.65$    &  $1.03$    & $26.51$&  $1.33$    & $33.76$  \\
      $931$ &$13.46$   &    $0.63$    & $18.64$ &    $0.66$    & $21.67$\\
      $1336$ &$18.27$   &    $0.48$    & $16.29$ &    $0.53$    & $18.26$\\
       $1741$&$23.08$   &   $0.42$    & $14.97$ &    $0.48$    & $17.13$\\
       $2146$&$27.88$   &   $0.39$    &$14.40$ & $0.43$ & $15.39$  \\

\hline
\end{tabular}
\end{table}
\subsection{A randomized multiscale adaptive algorithm}

In this section, we discuss how to efficiently use randomized snapshots within adaptive algorithms.  We use the error indicators developed in~\cite{ Chung_adaptive14}.  First, we briefly recall these error estimators.  Let $V_i = H^1_0(\omega_i)$, define a linear functional $R_i(v)$ on $V_i$ by
\begin{equation}
R_i(v) =  \int_{\omega_i} fv - \int_{\omega_i} a\nabla u_{\text{ms}}\cdot \nabla v,
\end{equation}
where the norm of $R_i$ is defined as
\begin{equation}
\| R_i \|_{V_i^*} = \sup_{v\in V_i} \frac{ |R_i(v)| }{\| v\|_{V_i}}.
\label{eq:Rnorm}
\end{equation}
Here {$\|v\|_{V_i} =(\int_{\omega_i} \kappa(x) | \nabla v|^2 \, dx)^{\frac{1}{2}}$. } In~\cite{Chung_adaptive14} it is shown that
\begin{eqnarray}
\| u-u_H\|_V^2 &\leq& C_{\text{err}}  \sum_{i=1}^N \|R_i\|^2_{V_i^*}  (\lambda^{\omega_i}_{l_i+1})^{-1}, \label{eq:res2}
\end{eqnarray}
where $C_{\text{err}}$ is a uniform constant and
$\lambda^{\omega_i}_{l_i+1}$ denotes the $({l_i+1})$-th eigenvalue over coarse neighborhood $\omega_i$ that corresponds to the first eigenvector excluded from the construction of $V_{\text{off}}$.   We define the error indicator in each coarse neighborhood as follows,
\begin{equation*}
\eta^2_i =
\|R_i\|^2_{V_i^*}  (\lambda^{\omega_i}_{l^m_i+1})^{-1},\quad \,\text{ for } H^{-1}\text{-based residual.}
\end{equation*}
The pivotal issue to solve is to generate   additional linearly independent basis for a selected coarse neighborhood $\omega_i$ for the current iteration. Specifically, those extra basis are required to be linearly independent from the basis in the previous iteration. In what follows, we describe a possible solution to this issue using the residue of a series of random basis and their projection onto the offline space of the previous iteration.

 \begin{table}
   \caption{Local basis enrichment algorithm}
   \label{algorithm:LocalBasisEnrichmentAlgorithm}
 \begin{tabular}{r l}
 \hline\hline
 \\
    \textbf{Input}:&an index of the coarse nodes $I$ selected by the error indicator for enrichment, \\
&the local offline space $\Psi_{\omega_i}^{\text{rsnap}}$,
 buffer number $c_{\text{bf}}^{\omega_i}$, \\
&an additional local basis number  $c_{\text{nb}}^{\omega_i} $ for each $i\in I$. \\
    \textbf{output}: &an enriched local offline space $\Psi_{\omega_i}^{\text{rsnap}}$ corresponds to each nodes in $I$.\\
  1.& Generate $c_{\text{nb}}^{\omega_i}+c_{\text{bf}}^{\omega_i}$ random vectors $r_l$ and obtain randomized snapshots in $\omega_i^{+}$ (Eqn.~\eqref{eq:random bc}). \\
  &Denote as $\phi_1, \cdots, \phi_{c_{\text{nb}}^{\omega_i}+c_{\text{bf}}^{\omega_i}}$ ;  \\
2. & A modification of the random basis obtained from Step 1:
$\tilde\phi_i=\phi_i-\sum\limits_{j=1}^{N}\frac{\langle \phi_i,\psi_j\rangle_M}{\langle \psi_j,\psi_j\rangle_M}\phi_i,$\\
&where $\psi_1, \cdots, \psi_{N}$ denote a series of basis of $\Psi_{\omega_i}^{\text{rsnap}}$ excluded the constant one; \\
  3. & Obtain $c_{\text{nb}}^{\omega_i}$ offline basis by a spectral decomposition (Eqn.~\eqref{offeig1}), \\
  & next, add in a snapshot that represent the constant function on $\omega_i^+$, \\
  &and denote the resulting vectors as $\Psi_{\omega_i}^{\text{enrich}}$;\\

4. & $\Psi_{\omega_i}^{\text{rsnap}}\equiv\Psi_{\omega_i}^{\text{rsnap}}\cup \Psi_{\omega_i}^{\text{enrich}}$.\\

 \hline\hline
 \end{tabular}
 \end{table}

 \begin{remark} The Step 2 in Table~\ref{algorithm:LocalBasisEnrichmentAlgorithm} is to guarantee that the added local basis are independent from the previous local basis in   the $M$-norm as defined in the next section.  In the randomized snapshots, we have added the constant local basis manually to guarantee that the multiscale basis are included. However, this constant basis should be excluded in Step 2 since the constant is not in the spectral vectors and if it is added, we can get linear dependency.  \end{remark}

 The numerical results are displayed in Table~\ref{table:perm_cross_adaptive}. First, we take five $(5)$ basis per coarse node. Then, we apply the multiscale adaptive algorithm proposed in~\cite{Chung_adaptive14} and identify the coarse nodes index $I$ requiring more basis. Set $c^{\omega_i}_{nb}=2$ and $c^{\omega_i}_{bf}=1$ and follow Table~\ref{algorithm:LocalBasisEnrichmentAlgorithm}, next, we generate $c^{\omega_i}_{nb}+c^{\omega_i}_{bf}=3$ local random basis for those nodes and use Step 2 to get three new linearly independent basis. Afterwards, a local spectral decomposition is performed to select two important basis from those three basis.  In the end, the corresponding multiscale basis functions are constructed and added to the coarse space.

 Comparing Tables~\ref{table:Reverse of HCCResult1} and~\ref{table:perm_cross_adaptive}, we observe that the randomized adaptive algorithm is  cheaper since much fewer basis functions are used to achieve comparable accuracy to that of the uniform increase of basis shown in Table~\ref{table:Reverse of HCCResult1}.  $2146$ basis functions are calculated to attain an energy error of $14.16\%$ in Table \ref{table:Reverse of HCCResult1}, while only $2061$ are necessary to get a smaller error of $13.90\%$ using the adaptive randomized algorithm. Here, we do not  discuss the computational cost of our adaptive algorithm and refer to \cite{Chung_adaptive14} for details. Our main goal in this section is simply to demonstrate how additional basis functions can be computed using a small set of new snapshots  that avoids storing all the eigenvectors. 
\begin{table}[htb!]
\centering
\caption{Numerical results using adaptive algorithm with $p_{\text{bf}}=4$, and  $5$ local basis per node at the beginning and   with two more basis for selected nodes,  $\kappa$ as shown in Fig.~\ref{fig:perm_cross}. }
\label{table:perm_cross_adaptive}
\small
\begin{tabular}{|r|c|c|}
\hline
\multirow{2}{*}{$\text{dim}(V_{\text{off}})$}  &
\multicolumn{2}{c|}{  using the usual snapshots (\%) }
 \\
\cline{2-3} {}&
$\hspace*{0.8cm}   L^{2}_\kappa(D)   \hspace*{0.8cm}$ &
$\hspace*{0.8cm}   H^{1}_\kappa(D)  \hspace*{0.8cm}$

\\
\hline\hline
       $526$     &  $4.11$    & $50.23$ \\
      $916$    &    $0.99$    & $21.65$ \\
      $1323$ &    $0.63$    & $17.33$ \\
       $1717$   &   $0.53$    & $15.10$\\
       $2061$&  $0.51$    & $13.90$\\

\hline
\end{tabular}
\end{table}


\section{Analysis}
\label{sec:analysis}
In the analysis described below, we first estimate the error due to the approximation using randomized snapshots.  In the first lemma, we compare an arbitrary snapshot obtained using all snapshot vectors and its approximation in the space of randomized snapshots. To avoid cumbersome notation, we denote the local snapshot matrix $\Psi_{\omega_i}^{\text{snap}}$ in (\ref{eq:snapdef})  by $\Psi$ and the local randomized snapshot matrix $\Psi_{\omega_i}^{\text{rsnap}}$ in~\eqref{eqn:random_snapshots} by $\Psi^{r}$.

The following lemma shows that the randomized snapshot $\Psi^{r}$ with $l$ random basis is a good approximation of the full snapshot $\Psi$ composed of $m$ basis, $m>l$. We use the notation $A\preceq B$ when $A \leq C B$ with $C$ being independent of    the size ratio between the coarse and fine meshes, and spatial scales. Throughout, $\|\cdot\|$ denotes the $l^2$ norm for vectors
and the $l^2$-based spectral norm for matrices, while $\|z\|_A=z^T A z$.
We remind that, 
throughout, for notational convenience, we do not distinguish between 
the fine-grid vectors and their continuous representations.

\begin{lemma}\label{lemma:1}
  Suppose $\Psi \in \mathbf{R}^{m\times n}$   of rank $m$, and $\mathcal{R}\in \mathbf{R}^{l\times m}$  whose entries are i.i.d. Gaussian random variables. Define $\Psi^{r}= \mathcal{R} \Psi$, then, for any $ \xi\in  \mathbf{R}^{m}$, there exists $\xi^{r}\in   \mathbf{R}^{l}$, such that
\begin{align}
\norm{\xi^{T}\Psi-(\xi^{r})^{T}\Psi^{r}}_{\widetilde{M}(\omega_i)}^{2}&=\int_{\omega_i} \widetilde{\kappa}|\xi^{T}\Psi-(\xi^{r})^{T}\Psi^{r}|^2
\preceq \left(\frac{\norm{\mathcal{H}^{(-1)}\mathcal{S}}+1}{\lambda_{k+1}}\right)^2\norm{\xi^{T}\Psi}_{A(\omega_i)}^2,&
\end{align}
where  $k<l<m<n$, and $\mathcal{S}$, $\mathcal{H}$, and $T$ are defined in Eqns.~\eqref{eq:Res} and~\eqref{eq:T}.

Here, $\lambda_{k+1}$ is the $(k+1)^{th}$ smallest diagonal value of $\Lambda$ defined in Eqn.~\eqref{eqn:eigen_value} and $A(\omega_{i})=(\int_{\omega_i}\;\kappa \nabla \phi_{j}\nabla \phi_{k})_{n\times n}$ with $\phi_j$ as the $j^{th}$ local fine-scale basis in the $\omega_i$. Besides, $\norm{\xi^{T}\Psi}_{\widetilde{M}(\omega_i)}=(\int_{\omega_i}\widetilde{\kappa} \xi^{T}\Psi \xi^{T}\Psi)^{1\over 2}$, $\norm{\xi^{T}\Psi}_{A(\omega_i)}=(\int_{\omega_i} \kappa \nabla(\xi^{T}\Psi)\cdot\nabla( \xi^{T}\Psi))^{1\over 2}$.  \end{lemma} 
\begin{proof}
Denote $\widetilde{M}(\omega_{i})=(\int_{\omega_i}\;\widetilde{\kappa} \phi_{j}\phi_{k})_{n\times n}$ with $\phi_j$ as the $j^{th}$ local fine-scale basis in $\omega_i$.  The matrix $\widetilde{M}(\omega_{i})$ is symmetric positive definite. Besides, $A(\omega_{i})$ is symmetric semi positive definite. Thus, there exists an $m\times m$ matrix $U$, such that
\begin{align}\label{eqn:eigen_value}
U^{T}\Psi \widetilde{M}(\omega_{i})\Psi^{T} U=\Lambda, \text{ and } U^{T}\Psi A(\omega_{i})\Psi^{T}U=I,
\end{align}
where $I$ is an identity matrix and $\Lambda$ denotes a diagonal matrix with decreasing diagonal values
\[
\displaystyle\frac{1}{\lambda_1}, \frac{1}{\lambda_2},\cdots, \frac{1}{\lambda_m}.
\]
Define $X=U^{-T}\Lambda^{1\over 2}$, then we obtain $XX^{T}=\Psi \widetilde{M}(\omega_{i})\Psi^{T}$.

Suppose $F$ is a matrix of dimension $m\times l$, take $\xi^{r}=F^{T}\xi$. Then
\begin{align*}
\int_{\omega_i} \widetilde{\kappa} |\xi^{T}\Psi-{\xi^{r^{T}}}\Psi^{r}|^2
&=(\xi^{T}\Psi-(\xi^{r})^{T} \mathcal{R} \Psi)\widetilde{M}(\omega_{i})
(\xi^{T}\Psi-(\xi^{r})^{T} \mathcal{R} \Psi)^{T}&\\
&=\xi^{T}(I-F \mathcal{R})\Psi \widetilde{M}(\omega_{i})\Psi^{T}(I-F\mathcal{R})^{T}\xi&\\
&=\xi^{T}(X-F\mathcal{R}X)(X-F\mathcal{R}X)^{T}\xi=\norm{\xi^{T}(X-F\mathcal{R}X)}^{2}.&
\end{align*}
In the following, we construct a matrix $F$ that minimizes $\norm{\xi^{T}(X-F\mathcal{R}X)}$.   Following~\cite[Lemma 18]{Martinsson06}, we define
\[F=U^{-T}
\begin{pmatrix}
\mathcal{H}^{(-1)}\\ 0
\end{pmatrix},\]
where $\mathcal{H}$ and $\mathcal{S}$ are matrices of dimension $l\times k$ and $l\times (m-k)$ defined as,
\begin{align}\label{eq:Res}
&\mathcal{R} U^{-T}=
\begin{pmatrix}
\mathcal{H} & \mathcal{S}
\end{pmatrix},&\\
&\mathcal{H}^{(-1)}=(\mathcal{H}^{T}\mathcal{H})^{-1}\mathcal{H}^{T}.&
\end{align}
That is,  $\mathcal{H}$ is of rank $k$ and contains the first $k$ columns of $\mathcal{R} U^{-T}$ and $\mathcal{H}^{(-1)}$ is the pseudo-inverse  of $\mathcal{H}$.

We obtain
\begin{align*}
\xi^{T}(X-F\mathcal{R}X)=-\xi^{T} U^{-T}\left(
\begin{pmatrix}
\mathcal{H}^{(-1)}\\ 0
\end{pmatrix}
\begin{pmatrix}
\mathcal{H} & \mathcal{S}
\end{pmatrix}
-I\right)\Lambda^{1\over 2}.
\end{align*}
Furthermore,
\begin{align*}
\norm{\xi^{T}(X-F\mathcal{R}X)}\leq \norm{\xi^{T}U^{-T}}(\norm{\mathcal{H}^{(-1)}\mathcal{S}T}+\norm{T}),
\end{align*}
where $T$ is defined as
\begin{align}\label{eq:T}
\Lambda^{1\over 2}=
\begin{pmatrix}
S&0\\
0&T
\end{pmatrix}
.
\end{align}
Thus, the spectral norm of $T$ is bounded, that is, $\displaystyle\norm{T}\leq \frac{1}{\lambda_{k+1}}$.
Then, using standard properties of subordinated norms we have,
\begin{eqnarray}
\norm{\xi^{T}(X-F\mathcal{R}X)}&\leq&
\norm{\xi^{T}U^{-T}}(\norm{\mathcal{H}^{(-1)}\mathcal{S}T}+\norm{T})\\
&\leq&\norm{\xi^{T}U^{-T}}(\norm{\mathcal{H}^{(-1)}\mathcal{S}}+1)\norm{T}\\
&\leq& \frac{\norm{\mathcal{H}^{(-1)}\mathcal{S}}+1}{\lambda_{k+1}}\norm{\xi^{T}\Psi}_{A(\omega_i)}.
\end{eqnarray}
Here, to obtain the last step we have used the relation~\eqref{eqn:eigen_value} that implies
\begin{align}
\norm{\xi^{T}U^{-T}}&=({\xi^{T}U^{-T}\cdot (\xi^{T}U^{-T})^{T}})^{1\over 2}=({\xi^{T}U^{-T}U^{-1} \xi })^{1\over 2}\\
&=({\xi^{T}\Psi A(\omega_i)\Psi^{T} \xi })^{1\over 2}=\norm{\xi^{T}\Psi}_{A(\omega_i)}.
\end{align}
Hence,
\begin{align*}
\int_{\omega_i} \kappa |\nabla\chi|^2|\xi^{T}\Psi-{\xi^{r^{T}}}\Psi^{r}|^2
\leq (\frac{\norm{\mathcal{H}^{(-1)}\mathcal{S}}+1}{\lambda_{k+1}})^2\norm{\xi^{T}\Psi}_{A(\omega_i)}^2.
\end{align*}
The proof is complete.
\end{proof}
\begin{remark} {\bf Estimate for $\norm{\mathcal{H}^{(-1)}\mathcal{S}}$.}\label{rmk:res}
$U$ in the Lemma~\ref{lemma:1} is orthonormal with respect to the $A(\omega_i)-$inner product. If $U$ is an orthonormal matrix itself, then by~\cite[Lemma 18]{Martinsson06}, $\norm{\mathcal{H}^{(-1)}\mathcal{S}}\leq \sqrt{l}\beta\frac{1}{\lambda_{k+1}^2}$ for some positive number $\beta$ and given $k$. If $U$ is not orthonormal, then by applying the Gram-Schmidt process to the first $k$ columns of $U^{-T}$ (denoted as $V_1$) as well as the rest of the columns of it (denoted as $V_2$), we can obtain non-singular  triangular matrices $D_1$ and $D_2$, and $S_1$ and $S_2$ with $S_1^{T}S_1=I$ and $S_2^{T}S_2=I$, such that
\[U^{-T}=
\begin{pmatrix}
V_1&V_2
\end{pmatrix}\]
\[V_1=S_1D_1, \text{ and }V_2=S_2D_2.\]
Then
\[\mathcal{H}^{(-1)}\mathcal{S}=(\mathcal{H}^{T}\mathcal{H})^{-1}
\mathcal{H}^{T}\mathcal{S}=(V_1^{T}\mathcal{R}^{T}\mathcal{R}V_1)^{-1}V_1^{T}\mathcal{R}^{T}\mathcal{R}V_2,\]
and using the expressions for $V_1$ and $V_2$, we obtain
\[\mathcal{H}^{(-1)}\mathcal{S}=D_{1}^{-1}(S_1^{T}\mathcal{R}^{T}\mathcal{R}S_1)^{-1}(D_{1}^{-T}D_{1}^{T})
(S_1^{T}\mathcal{R}^{T}\mathcal{R}S_2)D_2.\]
Therefore, we have
\[\norm{\mathcal{H}^{(-1)}\mathcal{S}}\leq \norm{D_{1}^{-1}}\norm{(S_1^{T}\mathcal{R}^{T}\mathcal{R}S_1)^{-1}S_1^{T}R^{T}}\norm{\mathcal{R}S_2}\norm{D_2}.\]
Since the entries of $RS_1$ and $RS_2$ are i.i.d. Gaussian random variables of zero mean and unit variance, using~\cite[Lemma 14]{Martinsson06}, we get the estimate,
\[\norm{\mathcal{H}^{(-1)}\mathcal{S}}\leq \sqrt{2lm\beta^{2}\gamma^{2}+1}\norm{D_{1}^{-1}}\norm{D_2}\]
with probability not less than 
{\[1-{\frac{1}{\sqrt{2\pi(l-k+1)}}}{\left({\frac{e}{(l-k+1)\beta}}\right)^{l-k+1}} -{\frac{1}{4(\gamma^2-1)\sqrt{\pi m\gamma^2}}}{\left({\frac{2\gamma^2}{e^{\gamma^2-1}}}\right)^{m}},\]}
where $\beta$ and $\gamma$ are positive real numbers, $\gamma>1$.

Next, we note that
the $i$-th diagonal elements of $D_1$ and $D_2$ are the norms of $i$-th columns of $V_1$ and $V_2$.  Moreover, $\norm{D_{1}^{-1}}\norm{D_2}$ is the ratio of the largest diagonal element of $D_2$ and the smallest diagonal element of $D_1$.  Since $U^{-T} U^{-1}=\Psi A \Psi^T$ and $U^{-T} \Lambda U^{-1}=\Psi \widetilde{M} \Psi^T$, the estimate of $\norm{\mathcal{H}^{(-1)}\mathcal{S}}$ depends on the norms of the columns of $U^{-T}$ and, thus, depends on the contrast, in general.  In  the particular case, we assume that $\Psi A \Psi^T$ is a diagonal matrix with entries $\lambda_1\leq \lambda_2\leq ...\lambda_n$. In this case, $U^{-T}=U^{-1}=diag(\lambda_1^{-1/2},\lambda_2^{-1/2},...,\lambda_n^{-1/2})$ and $D_1=diag(\lambda_1^{-1/2},\lambda_2^{-1/2},...,\lambda_l^{-1/2})$, $D_2=diag(\lambda_{l+1}^{-1/2},\lambda_{l+2}^{-1/2},...,\lambda_n^{-1/2})$.  Then, it is easy to verify that $\norm{D_{1}^{-1}}\norm{D_2}=\lambda_{l}^{1/2}/\lambda_{l+1}^{1/2}$ in this case.  This estimate shows that the error can be sensitive on the choice of the eigenspace that is selected. In GMsFEM, we usually select the most important eigenvalues that are very small (see \cite{egw10}), thus, in general, a contrast-dependent situation can be avoided.  \end{remark}

In Lemma~\ref{lemma:1}, we have derived the approximation of the randomized snapshot space to the full snapshot space locally in each patch $\omega_i$. Next, we present the convergence the GMsFEM using randomized snapshots. The snapshots are obtained by multiplying the local snapshots $\Psi_{\omega_i}^{\text{snap}}$ with the corresponding partition of unity function $\chi_i$ (as in Eqn.~\eqref{eqn:global_offline}). To simplify notation we denote by $\Psi$ the full global snapshots (snapshots for all $\omega_i$'s) and by $\Psi^{r}$ the full randomized snapshots (snapshots for all $\omega_i$'s).

\begin{theorem}
Denote by $\Psi$ the snapshot matrix and by $\Psi^r$ the randomized snapshot matrix of dimension $m\times n$ and $l\times n$, respectively, and their ranks are $m$ and $l$, respectively. $\mathcal {R}$ is a matrix with i.i.d. Gaussian random entries and that $\Psi^r=\mathcal{R}\Psi$. Suppose $u_H$ is solved using the offline space formed using the snapshot matrix $\Psi^r$, and $u$ is the fine-scale solution of Eqn.\eqref{eq:original}, then we have
\begin{eqnarray}
\int_D \kappa  |\nabla(u-u_H)|^2\preceq \left( {1\over {\Lambda_{*}}}+ \left (\frac{1}{\Lambda_{*}}\right)^2(\norm{\mathcal{H}^{(-1)}\mathcal{S}}+1)^2\right)\int_D \kappa |\nabla u|^2 +  H^2 \int_D f^2,
\end{eqnarray}
where $\Lambda_*$ is defined in (\ref{eqn:lambda}) and $l<m<n$.
\end{theorem}
\begin{proof}
Denote $I^{\omega_i}$ and $I^{\omega_i}_{r}$ as arbitrary interpolants from the fine-scale to the space spanned by the rows of $\Psi$ and $\Psi^r$ on the coarse neighborhood $\omega_i$, respectively. Later, we choose a proper interpolant that reduces the error.  Taking into account that the GMsFEM solution, $u_H$, provides a minimal energy error, we have
  \begin{equation}
\label{eq:thmmain}
\begin{split}
    \int_D \kappa  |\nabla(u-u_H)|^2
    &\preceq \int_D \kappa  |\nabla(\sum_i \chi_i(u-I^{\omega_i}_{r}u))|^2\\
    &\preceq
    \sum_i\int_{\omega_i} \kappa  |\nabla(\chi_i (u-I^{\omega_i}u))|^2+\int_{\omega_i} \kappa  |\nabla(\chi_i (I^{\omega_i}_{r}u- I^{\omega_i}u))|^2.
  \end{split}
\end{equation}
Next, we use the inequalities
\begin{align}
\label{eq:Cacc1}
\int_{\omega_i} \kappa \chi^2_i |\nabla(u- I^{\omega_i}u)|^2&\preceq \int_{\omega_i} \widetilde{\kappa}  |(u-I^{\omega_i}u)|^2 + \left|\int_{\omega_i}f\chi_i^2  (u-I^{\omega_i}u)\right|,\\
\label{eq:Cacc2}
\int_{\omega_i} \kappa \chi^2_i  |\nabla(I^{\omega_i}_ru-I^{\omega_i}u)|^2&\preceq \int_{\omega_i} \widetilde{\kappa}  |(I^{\omega_i}_{r}u-I^{\omega_i}u)|^2,
\end{align}
where $\widetilde{\kappa}$ is defined by~\eqref{def:tildekappa}. Here, we have used
the inequality (29) in \cite{egw10}. Using~\eqref{eq:Cacc1} and~\eqref{eq:Cacc2}, and we obtain from (\ref{eq:thmmain})
\begin{align}
 \int_D \kappa  |\nabla(u-u_H)|^2\preceq \sum_i \int_{\omega_i} \widetilde{\kappa}  |(u-I^{\omega_i}u)|^2 &+ \sum_i  \left|\int_{\omega_i}f\chi_i^2  (u-I^{\omega_i}u)\right| \nonumber \\
&+\sum_i \int_{\omega_i} \widetilde{\kappa}  |(I^{\omega_i}_{r}u-I^{\omega_i}u)|^2.
\end{align}
Selecting a proper interpolant $I^{\omega_i}$, we have
\begin{equation}
\label{eq:estim11}
\int_{\omega_i} \widetilde{\kappa}  |(u-I^{\omega_i}u)|^2 \preceq {1\over \lambda_{k+1}^{\omega_i}}\int_{\omega_i} \kappa |\nabla (u-I^{\omega_i}u)|^2,
\end{equation}
where $\lambda_{k+1}^{\omega_i}$ is the eigenvalue that the corresponding eigenvector  which is not included in the coarse space. Similarly, we can show that
\begin{align}
\left|\int_{\omega_i}f\chi_i^2  (u-I^{\omega_i}u)\right|&\preceq \int_{\omega_i}\widetilde{\kappa}^{-1} f^2 + \int_{\omega_i}\widetilde{\kappa} |(u-I^{\omega_i}u)|^2\nonumber\\ &\preceq
\int_{\omega_i}\widetilde{\kappa}^{-1} f^2 +  {1\over \lambda_{k+1}^{\omega_i}}\int_{\omega_i} \kappa |\nabla (u-I^{\omega_i}u)|^2.
\end{align}
We note that $\int_{\omega_i}\widetilde{\kappa}^{-1} f^2 \preceq H^2 \int_{\omega_i}f^2$
if $|\nabla \chi_i|=\mathcal{O}( H^{-1})$.
Combining the above estimates, we have
\begin{align}
  \int_D \kappa |\nabla(u-u_H)|^2 &\preceq
 \sum_i{1\over{\lambda_{k+1}^{\omega_i}}}   \int_{\omega_i} \kappa |\nabla (u-I^{\omega_i}u)|^2 \nonumber\\
&+\sum_i\int_{\omega_{i}}\widetilde{\kappa}^{-1}f^2+\sum_i\int_{\omega_i} \widetilde{\kappa}  |(I^{\omega_i}_{r}u-I^{\omega_i}u)|^2.
\label{eq:estimate1}
\end{align}
For a fixed vector $I^{\omega_i}u\in\Psi$, by Lemma~\ref{lemma:1}, we can get a corresponding vector $\xi^{r}\in\Psi^r$, such that
\begin{align}\label{eq:estimate2}
\int_{\omega_i} \widetilde{\kappa} |\xi^{r}-I^{\omega_i}u|^2\preceq \left(\frac{\norm{\mathcal{H}^{(-1)}(\omega_i)\mathcal{S}
(\omega_i)}+1}{\lambda_{k+1}^{\omega_i}}\right)^2\norm{I^{\omega_i}u}_{A(\omega_i)}^2,
\end{align}
for some integer $k$. For simplicity, we assume that $\lambda_{k+1}^{\omega_i}$ is the same eigenvector as in the interpolant defined in (\ref{eq:estim11}) by selecting the smallest index.

We define $I^{\omega_i}_{r}u=\xi^{r}$.
Thus using Eqns.~\eqref{eq:estimate1} and~\eqref{eq:estimate2}, we obtain
\begin{align}
  \int_D \kappa |\nabla(u-u_H)|^2 &\preceq \max\limits_{\omega_i}\left( {1\over{\lambda_{k+1}^{\omega_i}}} \right)
  \int_{\omega_i} \kappa |\nabla u|^2 + \sum_i\int_{\omega_{i}}\widetilde{\kappa}^{-1}f^2\nonumber\\ &+\sum_i\left(\frac{\norm{\mathcal{H}^{(-1)}\mathcal{S}}+1}{\lambda_{k+1}^{\omega_i}}\right)^2\norm{I^{\omega_i}u}_{A(\omega_i)}^2
  \\ &\preceq
  \left({1\over \Lambda_*} + {1\over \Lambda_*^2}(\norm{\mathcal{H}^{(-1)}\mathcal{S}}+1)^2 \right) \int \kappa |\nabla u|^2 + \sum_i\int_{\omega_{i}}\widetilde{\kappa}^{-1}f^2,
\end{align}
where
 \begin{align}\label{eqn:lambda}
  \Lambda_{*}=\min_{\omega_i}{{\lambda_{k+1}^{\omega_i}}}.
  \end{align}
Here, we have used the boundedness property of the interpolant in the energy norm \cite{ge09_1}.
Assuming $|\nabla \chi_i|=\mathcal{O}( H^{-1})$, we get
\begin{equation}
\begin{split}
\int_D \kappa  |\nabla(u-u_H)|^2 \preceq
\left({1\over \Lambda_*} + {1\over \Lambda_*^2}(\norm{\mathcal{H}^{(-1)}\mathcal{S}}+1)^2  \right) \int \kappa |\nabla u|^2 +
 H^2 \int_{D}f^2.
 \end{split}
\end{equation}

 \end{proof}

\begin{remark}
One can improve the error due to GMsFEM discretization by changing the eigenvalue
problem (see \cite{eglp13oversampling}) and the error will scale as ${1\over \Lambda_*^q}$, for a large $q$ that depends on the size of the oversampled region.
In this case, the error due to GMsFEM discretization
will scale as $(1/\Lambda_*)^n$ for some large $n$.
 \end{remark}

\section{Conclusions}
\label{sec:conclusions}
In this paper, we study the use of randomized boundary conditions to reduce the computational cost in multiscale finite element methods. Local multiscale finite element basis functions are constructed in each coarse patch by computing snapshot vectors and performing local spectral decompositions.  The choice of snapshot vectors and the local spectral decomposition is important for achieving a low dimensional coarse spaces that can approximate the solution accurately on a coarse mesh.  For example, the use of harmonic functions computed in oversampled regions improves the accuracy. However, the computation of harmonic functions for all possible boundary conditions in  each local region is expensive. Therefore, we propose the use of randomized boundary conditions for computing the snapshot vectors.  We show that with a few snapshot vectors, we can compute the basis functions that provide an accuracy that is similar to that obtained using all snapshot vectors. We analyze  the method and validate our estimates with numerical evidence.  Moreover, we discuss approaches that are more accurate compared to randomized snapshot; however, they are more expensive. Finally, we discuss how adaptive computations can be performed efficiently and robustly within the framework of randomized snapshots where multiscale basis functions are added locally in some regions based on an error indicator.

\section*{References}
\bibliographystyle{plain}
\bibliography{references}
\end{document}